\documentclass[11pt,a4paper]{amsart}
\usepackage[margin=2.2cm]{geometry}
\usepackage{amssymb}
\usepackage{graphicx} 
\usepackage[mathscr]{euscript}
\usepackage{enumerate}
\usepackage{xspace}
\usepackage{color}
\usepackage{enumerate}
\usepackage{enumitem}
\usepackage{amsthm}
\usepackage{dsfont}
\usepackage{bbm}
\usepackage[colorlinks=true, linkcolor = blue, citecolor = blue]{hyperref}
\usepackage[numbers]{natbib}
\usepackage[backgroundcolor=white,bordercolor=red]{todonotes}
\DeclareMathAlphabet{\mathpzc}{OT1}{pzc}{m}{it}
\usepackage[nameinlink]{cleveref}
\numberwithin{equation}{section}
\begin{document}

\theoremstyle{plain}

\newtheorem{theorem}{Theorem}[section]
\newtheorem{lemma}[theorem]{Lemma}
\newtheorem{example}[theorem]{Example}
\newtheorem{proposition}[theorem]{Proposition}
\newtheorem{corollary}[theorem]{Corollary}
\newtheorem{definition}[theorem]{Definition}
\newtheorem{Ass}[theorem]{Assumption}
\newtheorem{condition}[theorem]{Condition}
\theoremstyle{definition}
\newtheorem{remark}[theorem]{Remark}
\newtheorem{SA}[theorem]{Standing Assumption}

\newcommand{\of}{[\hspace{-0.06cm}[}
\newcommand{\gs}{]\hspace{-0.06cm}]}

\newcommand\llambda{{\mathchoice
		{\lambda\mkern-4.5mu{\raisebox{.4ex}{\scriptsize$\backslash$}}}
		{\lambda\mkern-4.83mu{\raisebox{.4ex}{\scriptsize$\backslash$}}}
		{\lambda\mkern-4.5mu{\raisebox{.2ex}{\footnotesize$\scriptscriptstyle\backslash$}}}
		{\lambda\mkern-5.0mu{\raisebox{.2ex}{\tiny$\scriptscriptstyle\backslash$}}}}}

\newcommand{\1}{\mathds{1}}

\newcommand{\F}{\mathbf{F}}
\newcommand{\G}{\mathbf{G}}

\newcommand{\B}{\mathbf{B}}

\newcommand{\M}{\mathcal{M}}

\newcommand{\la}{\langle}
\newcommand{\ra}{\rangle}

\newcommand{\lle}{\langle\hspace{-0.085cm}\langle}
\newcommand{\rre}{\rangle\hspace{-0.085cm}\rangle}
\newcommand{\blle}{\Big\langle\hspace{-0.155cm}\Big\langle}
\newcommand{\brre}{\Big\rangle\hspace{-0.155cm}\Big\rangle}

\newcommand{\X}{\mathsf{X}}

\newcommand{\tr}{\operatorname{tr}}
\newcommand{\N}{{\mathbb{N}}}
\newcommand{\cadlag}{c\`adl\`ag }
\newcommand{\on}{\operatorname}
\newcommand{\oP}{\overline{P}}
\newcommand{\oO}{\mathcal{O}}
\newcommand{\D}{D(\mathbb{R}_+; \mathbb{R})}

\renewcommand{\epsilon}{\varepsilon}

\newcommand{\fPs}{\mathfrak{P}_{\textup{sem}}}
\newcommand{\fPas}{\mathfrak{P}^{\textup{ac}}_{\textup{sem}}}
\newcommand{\rrarrow}{\twoheadrightarrow}
\newcommand{\cC}{\mathcal{C}}
\newcommand{\cH}{\mathcal{H}}
\newcommand{\cD}{\mathcal{D}}
\newcommand{\cE}{\mathcal{E}}
\newcommand{\cR}{\mathcal{R}}
\newcommand{\cQ}{\mathcal{Q}}
\newcommand{\cF}{\mathcal{F}}
\newcommand{\bth}{\overset{\leftarrow}\theta}
\renewcommand{\th}{\theta}

\newcommand{\bR}{\mathbb{R}}
\newcommand{\nnabla}{\nabla}
\newcommand{\f}{\mathfrak{f}}
\newcommand{\g}{\mathfrak{g}}
\newcommand{\oconv}{\overline{\operatorname{conv}}\hspace{0.1cm}}
\newcommand{\usa}{\on{usa}}
\newcommand{\usc}{\textit{USC}}
\newcommand{\C}{\mathsf{C}}
\newcommand{\ou}{\overline{u}}
\newcommand{\ua}{\underline{a}}
\newcommand{\uu}{\underline{u}}
\newcommand{\cK}{\mathcal{K}}

\renewcommand{\emptyset}{\varnothing}

\makeatletter
\@namedef{subjclassname@2020}{%
	\textup{2020} Mathematics Subject Classification}
\makeatother

 \title[]{A Class of Multidimensional Nonlinear Diffusions \\ with the Feller Property}
\author[D. Criens]{David Criens}
\author[L. Niemann]{Lars Niemann}
\address{Albert-Ludwigs University of Freiburg, Ernst-Zermelo-Str. 1, 79104 Freiburg, Germany}
\email{david.criens@stochastik.uni-freiburg.de}
\email{lars.niemann@stochastik.uni-freiburg.de}



\thanks{DC acknowledges financial support from the DFG project SCHM 2160/15-1 and LN acknowledges financial support from the DFG project SCHM 2160/13-1.}
\date{\today}

\maketitle

\begin{abstract}
In this note we consider a family of nonlinear (conditional) expectations that can be understood as a multidimensional diffusion with uncertain drift and certain volatility. Here, the drift is prescribed by a set-valued function that depends on time and path in a Markovian way.
We establish the Feller property for the associated sublinear Markovian semigroup and we observe a smoothing effect as our framework carries enough randomness. 
Furthermore, we link the corresponding value function to a semilinear Kolmogorov equation.
\end{abstract}

\section{Introduction}
A \emph{nonlinear multidimensional diffusion}, or \emph{nonlinear multidimensional continuous Markov process}, is a family of sublinear expectations \( \{\cE^x \colon x \in \bR^d\} \) on the Wiener space \( C(\bR_+; \bR^d) \) with
\( \cE^x \circ X_0^{-1} = \delta_x \) for each \( x \in \bR^d \) such that the Markov property
\begin{equation} \label{eq: markov property}
\cE^x(\cE^{X_t}(\psi (X_s))) = \cE^x(\psi (X_{t+s}) ), \quad x \in \bR^d, \ s,t \in \bR_+,
\end{equation}
holds. Here, \(\psi\) runs through a collection of suitable test functions and \( X \) denotes the canonical process on \( C(\bR_+; \bR^d) \).
Building upon the seminal work of Peng \cite{peng2007g, peng2008multi} on the \(G\)-Brownian motion, nonlinear Markov processes have recently been studied from the perspective of processes under uncertainty, see \cite{fadina2019affine, hu2021g, K19, neufeld2017nonlinear}.
Using the techniques from \cite{NVH}, a general framework for constructing nonlinear Markov processes was developed in \cite{hol16}.
To be more precise, for given \( x \in \bR\), the sublinear expectation
\( \cE^x \) has the form
\( \cE^x = \sup_{P \in \cR(x)} E^P \)
with a collection \( \cR(x) \) of semimartingale laws \(P\) on the path space, with initial distribution \( \delta_x \) and absolutely continuous semimartingale characteristics \((B^{P}, C^{P})\), where the differential characteristics \((dB^{P} /d\llambda, dC^{P}/d\llambda)\) are prescribed in a Markovian way.

As in the theory of (linear) Markov processes, there is a strong link to semigroups. Indeed, the Markov property \eqref{eq: markov property} ensures the semigroup property \( T_t T_s = T_{s+t}, \ s,t \in \bR_+ \), where the sublinear operators \( T_t, \ t \in \bR_+, \) are defined by 
\begin{equation} \label{eq: def semigroup}
    T_t(\psi)(x) := \cE^x(\psi(X_t)) = \sup_{P \in \cR(x)} E^P \big[ \psi(X_t) \big]
\end{equation}
for suitable functions \(\psi\).
Using the general theory of \cite{ElKa15, NVH}, the operators
\(T_t, t \in \mathbb{R}_+\), are well-defined on the cone of upper semianalytic functions.

In our paper \cite{CN22b}, for a one-dimensional setting, we discovered a novel type of smoothing effect of the semigroup \((T_t)_{t \in \bR_+}\) under a weak ellipticity assumption. Namely, we proved that \((T_t)_{t \in \bR_+}\) has the so-called \emph{strong \(\usc_b\)--Feller property}, i.e., for every \(t > 0\), the operator \(T_t\) maps bounded upper semicontinuous functions to bounded continuous functions. 
Further, under a uniform ellipticity condition, we even established the {\em uniform} strong \(\usc_b\)--Feller property, that shows that each \(T_t\), for times \(t > 0\), maps bounded upper semicontinuous functions to bounded uniformly continuous functions. 
To the best of our knowledge, smoothing effects of these specific forms were not reported before. Related, but generally different, smoothing effects are known for viscosity solutions to parabolic Hamilton--Jacobi--Bellman (or more general nonlinear) PDEs, see, e.g., \cite{Cra02,kry18, kry17} and the references therein.

In this note, we establish the uniform strong \(\usc_b\)--Feller property for a class of {\em multidimensional} diffusions with \emph{uncertain} drift \(b \colon F \times \bR^d \to \bR^d\) and \emph{certain} volatility \(a \colon \bR^d \to \mathbb{S}^{d}_+\). 
Furthermore, leaning on theory from \cite{CN22,CN22b,hol16}, the (uniform) strong \(\usc_b\)--Feller property allows us, for \(\psi \in C_b(\bR^d; \bR)\), to identify the so-called value function \((t,x) \mapsto T_t(\psi)(x)\)
as a bounded viscosity solution to the nonlinear Kolmogorov type PDE
\begin{equation} \label{eq: intro PDE}
\begin{cases}   
\partial_t u (t, x) - G (x, u(t,\cdot \,)) = 0, & \text{for } (t, x) \in \bR_+ \times \mathbb{R}^d, \\
u (0, x) = \psi (x), & \text{for } x \in \bR^d,
\end{cases}
\end{equation}
where
\begin{align*}
    G(x, \phi) := \sup_{f \in F} \Big\{ \langle b (f, x), \nabla \phi (x) \rangle \Big\}
    + \tfrac{1}{2} \on{tr} \big[a (x) \nabla^2 \phi (x) \big].
\end{align*}
Under additional Lipschitz conditions, the value function is even the unique bounded viscosity solution to \eqref{eq: intro PDE}. Thanks to results from \cite{hol16,K19,K21}, this uniqueness result transfers to the level of semigroups. Namely, it shows that \((T_t)_{t \in \bR_+}\) is the only jointly continuous sublinear semigroup on \(C_b (\bR^d; \bR)\) whose pointwise generator coincides with \(G\) on the set \(C^\infty_c (\bR^d; \bR)\) of smooth functions with compact support.
For sublinear convolution semigroups, such a stochastic representation has been obtained in \cite{K21}. For a statement on the level of nonlinear Markov processes with jumps, we refer to our paper \cite{CN23a} that was finished after the current paper has been submitted.

Sublinear semigroups can also be constructed by analytic methods, see \cite{denk2020semigroup, hol16, NR}. A general approach leading to the so-called {\em Nisio semigroup} and a corresponding viscosity theory was recently established in the paper \cite{NR}. The framework from \cite{NR} allows for general state spaces, and provides conditions for Feller properties on spaces of weighted continuous functions. If the weight function is vanishing at infinity, this includes the \(C_b\)--Feller property, i.e., the semigroup is a self-map on the space of bounded continuous functions. Further, in case the weight function is bounded from below and the Nisio semigroup is continuous from above (in a suitable sense), the \(C_b\)--Feller property can be derived, too.  Specific examples for nonlinear Markov processes whose associated Nisio semigroups entail the \(C_b\)--Feller property were discussed in \cite[Section 6.3]{NR}.
For the case of convolution semigroups, corresponding to the L\'evy framework from \cite{denk2020semigroup, neufeld2017nonlinear}, it has been shown in \cite{K21} that \((T_t)_{t \in \bR_+}\) coincides with the Nisio semigroup. A similar result for certain one-dimensional nonlinear diffusions is given in the update of our paper~\cite{CN22b}.
It appears to us that the analytic treatment provides currently no access to the (uniform) strong \(\usc_b\)--Feller property.

\section{Main Result}
\subsection{The Setting}\label{subsec: setting}
Fix a dimension \(d \in \mathbb{N}\) and define $\Omega$ to be the space of continuous functions \(\mathbb{R}_+ \to \mathbb{R}^d\) endowed with the local uniform topology. 
The canonical process on $\Omega$ is denoted by \(X\), i.e., \(X_t (\omega) = \omega (t)\) for \(\omega \in \Omega\) and \(t \in \mathbb{R}_+\). 
It is well-known that \(\mathcal{F} := \mathcal{B}(\Omega) = \sigma (X_t, t \in \bR_+)\).
We define $\F := (\mathcal{F}_t)_{t \in \bR_+}$ as the canonical filtration generated by $X$, i.e., \(\mathcal{F}_t := \sigma (X_s, s \in [0, t])\) for \(t \in \mathbb{R}_+\). 
The set of probability measures on \((\Omega, \mathcal{F})\) is denoted by \(\mathfrak{P}(\Omega)\) and endowed with the usual topology of convergence in distribution.
We denote the space of symmetric positive semidefinite real-valued \(d\times d\) matrices by \(\mathbb{S}^d_{+}\).
Let \(F\) be a metrizable space and let \(b \colon F \times \mathbb{R}^d \to \mathbb{R}^d\) and \(a \colon \mathbb{R}^d \to \mathbb{S}^d_{+}\) be two Borel functions.

\begin{condition} \label{SA: bounded and elliptic}
\quad 
\begin{enumerate}
\item[\textup{(i)}] \(F\) is compact.
    \item[\textup{(ii)}] \(b\) and \(a\) are continuous.
\item[\textup{(iii)}]
There exists a constant \(\C > 0\) such that, for all \(f \in F\) and \(x, \xi \in \mathbb{R}^d\),
\[
\| b( f, x ) \| \leq \C, \qquad \frac{\|\xi\|^2}{\C} \leq \langle \xi, a(x) \xi \rangle \leq \C \|\xi\|^2.
\]
\end{enumerate}
\end{condition}

\begin{remark}
In the this note, we will always work under Condition~\ref{SA: bounded and elliptic}.
While it is possible so substantially weaken part (iii), we assume it for the sake of clarity. 
\end{remark}
We define the correspondence, i.e., the set-valued map, \(\Theta \colon \bR^d \twoheadrightarrow \mathbb{R}^d \times \mathbb{S}^d_{+}\) by
\[
\Theta (x) := \big\{(b (f, x), a (x)) \colon f \in F \big\} \subset \mathbb{R}^d \times \mathbb{S}^d_{+}.
\]

\begin{remark}
Thanks to \cite[Lemma 2.9]{CN22}, under Condition~\ref{SA: bounded and elliptic}, the graph of \(\Theta\) is measurable.
\end{remark}

We denote the set of laws of continuous semimartingales by \(\fPs \subset \mathfrak{P}(\Omega)\).
For \(P \in \fPs\), we denote the semimartingale characteristics of the coordinate process \(X\) by \((B^P, C^P)\), and 
we set 
\[
\fPas  := \big\{ P \in \fPs  \colon P\text{-a.s. } (B^P, C^P) \ll \llambda \big\},
\]
where \(\llambda\) denotes the Lebesgue measure.
For \( x \in \bR^d \), we further define
\[
\cR (x) := \big\{ P \in \fPas \colon P \circ X_0^{-1} = \delta_{x},\ (\llambda \otimes P)\text{-a.e. } (dB^{P} /d\llambda, dC^{P}/d\llambda) \in \Theta(X)   \big\}.
\]

\begin{remark}
    By virtue of \cite[Lemma 2.10]{CN22}, under Condition~\ref{SA: bounded and elliptic}, the set \(\cR(x)\) is non-empty for every \(x \in \bR^d\).
\end{remark}


\subsection{Nonlinear Diffusions and Sublinear Semigroups} \label{subsec: markovian semigroups}

For each \( x \in \bR^d \), we define the sublinear operator \( \cE^x \) on the convex cone of upper semianalytic functions  \(\psi \colon \Omega \to \bR \) by
\[ \cE^x(\psi) := \sup_{P \in \cR(x)} E^P \big[ \psi \big]. \] 
For every \( x \in \bR^d \), we have by construction that \( \cE^x(\psi(X_0)) = \psi(x) \) for every bounded upper semianalytic function \( \psi \colon \bR^d \to \bR \).

The next proposition confirms that the family \(\{\cE^x \colon x \in \bR^d\}\)
is a nonlinear multidimensional diffusion. It provides the so-called nonlinear Markov property of the family
\( \{ \cE^x \colon x \in \bR^d\}\), cf. \cite[Lemma 4.32]{hol16} and \cite[Proposition 2.8]{CN22b}. We omit a detailed proof.

\begin{proposition}  [Nonlinear Markov Property] \label{prop: markov property}
Suppose that Condition~\ref{SA: bounded and elliptic} holds.
For \(t \in \mathbb{R}_+\),  denote the shift operator \(\theta_t \colon \Omega \to \Omega\) by \(\theta_t (\omega) := \omega(\,\cdot + t)\).
Then, for every upper semianalytic function \( \psi \colon \Omega \to [- \infty, \infty] \), the equality
\[
\cE^x( \psi \circ \theta_t) = \cE^x ( \cE^{X_t} (\psi))
\]
holds for every \((t, x) \in \bR_+ \times \bR^d\).
\end{proposition}

\begin{definition}
Let \( \mathcal{H} \) be a convex cone of functions \( f \colon \bR^d \to \bR \) containing all constant functions.
A family of sublinear operators \( T_t \colon \mathcal{H} \to \mathcal{H}, \ t \in \bR_+,\) is called a \emph{sublinear Markovian semigroup} on \( \mathcal{H} \) if it satisfies the following properties:
\begin{enumerate}
    \item[\textup{(i)}] \( (T_t)_{t \in \bR_+} \) has the semigroup property, i.e., 
          \( T_s T_t = T_{s+t} \) for all \(s, t \in \bR_+ \) and
          \( T_0 = \on{id} \),
    
    \item[\textup{(ii)}] \( T_t \) is monotone for each \( t \in \bR_+\), i.e., 
    \( f, g \in \mathcal{H} \) with \( f \leq g \) implies \(T_t (f) \leq T_t (g) \),
    
    \item[\textup{(iii)}] \( T_t \) preserves constants for each  \( t \in \bR_+\), i.e.,
    \( T_t(c) = c \) for each \( c \in \bR  \).
\end{enumerate}
\end{definition}

The following proposition should be compared to \cite[Remark 4.33]{hol16} and \cite[Proposition 2.8]{CN22b}. 
For brevity, we omit a detailed proof. 

\begin{proposition}[Semigroup Property]
Suppose that Condition~\ref{SA: bounded and elliptic} holds.
The family of operators \( (T_t)_{t \in \bR_+} \) given by
\begin{align} \label{eq: def T} T_t ( \psi )(x) := \cE^x(\psi(X_t)), \quad t \in \bR_+, \ x \in \bR^d, \end{align}
defines a sublinear Markovian semigroup on the set of bounded upper semianalytic functions.
\end{proposition}

\subsection{The Feller Property}

In the following theorem, which is the main result of this note, we show that \((T_t)_{t \in \mathbb{R}_+}\) is a nonlinear semigroup on the space \(C_b(\mathbb{R}^d; \mathbb{R})\) of bounded continuous functions from \(\mathbb{R}^d\) into \(\mathbb{R}\). In fact, we show a bit more, namely the existence of a so-called \emph{uniform strong Feller selection} (see Definition~\ref{def: USFF} below).

\begin{condition} \label{cond: convex}
The set \(
\{ b (f, x) \colon f \in F \} \subset \mathbb{R}^d
\) is convex for every \(x \in \mathbb{R}^d\).
\end{condition}
\begin{theorem} \label{thm: main}
Suppose that the Conditions~\ref{SA: bounded and elliptic} and \ref{cond: convex} hold.
Then, for every \(t > 0\) and any bounded upper semicontinuous function \(\psi \colon \mathbb{R}^d \to \mathbb{R}\), the map \(x \mapsto T_t (\psi)(x)\) is uniformly continuous. 
In particular, \((T_t)_{t \in \mathbb{R}_+}\) has the \(C_b\)--Feller property, i.e., 
\(T_t(C_b(\bR^d;\bR)) \subset C_b(\bR^d;\bR)\).
\end{theorem}

\begin{remark}
Theorem \ref{thm: main} shows that \((T_t)_{t \in \mathbb{R}_+}\) has a weak version of the uniform strong Feller property (Definition~\ref{def: USFF} below), i.e., that regularity is gained through \(T_t\), with \(t > 0\), since bounded \emph{upper semicontinuous} functions are mapped to bounded \emph{uniformly continuous} functions. 
\end{remark}

For a {\em one}-dimensional continuous nonlinear framework with uncertain volatility, i.e., with \(F\)-dependent diffusion coefficient \(a\), a version of Theorem \ref{thm: main} was proved in our previous paper \cite{CN22b}.

\subsection{An Application to Semilinear PDEs}
For \((x,\phi) \in \times \mathbb{R}^d \times C^{2}(\mathbb{R}^d; \bR)\), we define the nonlinearity
\begin{align} \label{eq: G main text}
    G(x, \phi) := \sup_{f \in F} \Big\{ \langle b (f, x), \nabla \phi (x) \rangle \Big\}
    + \tfrac{1}{2} \on{tr} \big[ a (x) \nabla^2 \phi (x) \big].
\end{align}
Recall that a function \(u \colon \bR_+ \times \bR^d \to \mathbb{R}\) is said to be a \emph{weak sense viscosity subsolution} to the semilinear PDE
\begin{equation} \label{eq: PDE}
\begin{cases}   
\partial_t u (t, x) - G (x, u(t, \cdot \,)) = 0, & \text{for } (t, x) \in \bR_+ \times \mathbb{R}^d, \\
u (0, x) = \psi (x), & \text{for } x \in \bR^d,
\end{cases}
\end{equation}
where \(\psi  \in C_b(\mathbb{R}^d; \mathbb{R})\),
if the following two properties hold:
\begin{enumerate}
    \item[\textup{(a)}] \(u(0, \cdot\,) \leq \psi\);
\item[\textup{(b)}]
\(
\partial_t \phi (t, x) - G (x, \phi(t,\cdot\,)) \leq 0
\)
for all \(\phi \in  C^\infty_b(\bR_+ \times \bR^d; \bR)\) such that \(u \leq \phi\) and \(\phi (t, x) = u(t, x)\) for some \((t, x) \in (0,\infty) \times \bR^d \). 
\end{enumerate}
A \emph{weak sense viscosity supersolution} is obtained by reversing the inequalities. Further, \(u\) is called \emph{weak sense viscosity solution} if it is a weak sense viscosity sub- and supersolution. 
Additionally,  \( u \) is called \emph{viscosity subsolution} if it is both, a weak sense viscosity subsolution, and upper semicontinuous. The notions of viscosity supersolution and viscosity solution are defined accordingly.

Similar as in our previous paper \cite{CN22}, using the Conditions~\ref{SA: bounded and elliptic} and \ref{cond: convex},
one can prove that the so-called value function 
\[ v (t, x) := T_t(\psi)(x), \quad (t,x) \in \bR_+ \times \bR^d,\]
is a weak-sense viscosity solution to \eqref{eq: PDE}. The uniform strong \(\usc_b\)--Feller property from Theorem~\ref{thm: main} yields additional regularity of the semigroup which can be used to show that \(v\) is even a viscosity solution in the classical sense. Under Lipschitz conditions on \(b\) and \(a\), we can even deduce a uniqueness statement.

\begin{condition}[Lipschitz Continuity in Space] \label{cond: Lipschitz continuity} 
    There exists a decomposition \(a = \sigma \sigma^*\) and a constant \(\C > 0\) such that 
    \[
    \|b(f, x) - b(f, y)\| + \|\sigma (x) - \sigma (y)\| \leq \C \|x - y\|, 
    \]
    for all \(f \in F\) and \(x, y \in \mathbb{R}^d\).
\end{condition}

\begin{theorem} \label{thm: main viscosity}
Suppose that the Conditions~\ref{SA: bounded and elliptic} and \ref{cond: convex} hold. Then, the value function \(v\) is a viscosity solution to the nonlinear PDE \eqref{eq: PDE}. If, in addition, Condition~\ref{cond: Lipschitz continuity} holds, then \(v\) is the unique bounded viscosity solution.
\end{theorem}

\begin{proof}
We already mentioned that \(v\) is a weak sense viscosity solution to \eqref{eq: PDE}. Furthermore, thanks to Theorem \ref{thm: main}, it follows verbatim as in the proof of \cite[Theorem 2.36]{CN22b} that \(v\) is continuous (in both arguments). Finally, the comparison principle \cite[Corollary~2.34]{hol16}, in combination with \cite[Lemmata~2.4, 2.6]{hol16} and \cite[Remark~2.5]{hol16}, implies  uniqueness under Condition~\ref{cond: Lipschitz continuity}.
\end{proof}

\section{Proof of Theorem \ref{thm: main}}
We call an \(\bR^d\)-valued continuous process \(Y = (Y_t)_{t \geq 0}\) a (continuous) \emph{semimartingale after a time \(t^* \in \mathbb{R}_+\)} if the process \(Y_{\cdot + t^*} = (Y_{t + t^*})_{t \geq 0}\) is a semimartingale for its natural right-continuous filtration.
The law of a semimartingale after \(t^*\) is said to be a \emph{semimartingale law after \(t^*\)} and the set of them is denoted by \(\fPs (t^*)\). 
For \(P \in \fPs (t^*)\) we denote the semimartingale characteristics of the shifted coordinate process \(X_{\cdot + t^*}\) by \((B^P_{\cdot + t^*}, C^P_{\cdot + t^*})\), 
and we set
\[
\fPas (t^*) := \big\{ P \in \fPs (t^*) \colon P\text{-a.s. } (B^P_{\cdot + t^*}, C^P_{\cdot + t^*}) \ll \llambda \big\}.
\]
For \((t,x) \in \bR_+ \times \bR^d\), we define 
\begin{align*}
\cK(t,x) := \big\{ P \in \mathfrak{P}_{\text{sem}}^{\text{ac}}(t)\colon P(X_s &= x \text{ for all } s \in [0, t]) = 1, \\
&(\llambda \otimes P)\text{-a.e. } (dB^P_{\cdot + t} /d\llambda, dC^P_{\cdot + t}/d\llambda) \in \Theta (X_{\cdot + t}) \big\}.
\end{align*}
For a probability measure \(P\) on \((\Omega, \mathcal{F})\), a kernel \(\Omega \ni \omega \mapsto Q_\omega \in \mathfrak{P}(\Omega)\), and a finite stopping time \(\tau\), we define the pasting measure
\[
(P \otimes_\tau Q) (A) \triangleq \iint \1_A (\omega \otimes_{\tau(\omega)} \omega') Q_\omega (d \omega') P(d \omega), \quad A \in \cF,
\]
where
\[
\omega \otimes_t \omega' :=  \omega \1_{[ 0, t)} + (\omega (t) + \omega' - \omega' (t))\1_{[t, \infty)}.
\]

\begin{definition}[Time inhomogeneous Markov Family]
	A family \(\{P_{(s, x)} \colon (s, x) \in \bR_+ \times \bR^d\} \subset \mathfrak{P}(\Omega)\) is said to be a \emph{strong Markov family} if \((t, x) \mapsto P_{(t, x)}\) is Borel and the strong Markov property holds, i.e., for every \((s, x) \in \bR_+ \times \bR^d\) and every finite stopping time \(\tau \geq s\), 
	\[
	P_{(s, x)} (\,\cdot\, | \cF_\tau) (\omega) = \delta_\omega \otimes_{\tau (\omega)} P_{(\tau (\omega), \omega (\tau (\omega)))}
	\]
for \(P_{(s, x)}\)-a.a. \(\omega \in \Omega\).
\end{definition}

The following general strong Markov selection principle can be proved as in Section 5 of \cite{CN22b}. We omit a detailed proof. 
\begin{theorem}[Strong Markov Selection Principle]  \label{theo: strong Markov selection}
Suppose that the Conditions~\ref{SA: bounded and elliptic} and \ref{cond: convex} hold.
	For every \(\psi \in \usc_b(\bR; \mathbb{R}^d)\) and every \(t > 0\), there exists a strong Markov family \(\{P_{(s, x)} \colon (s, x) \in \bR_+ \times \mathbb{R}^d\}\) such that, for all \((s, x)\in \bR_+ \times \mathbb{R}^d\), \(P_{(s, x)} \in \cK (s, x)\) and 
	\[
	E^{P_{ (s, x) }} \big[ \phi (X_t) \big] = \sup_{P \in \cK (s, x)} E^P \big[ \phi (X_t) \big].
	\]
	In particular, for all \(x \in \bR^d\), 
	\[
	T_t (\psi) (x) = E^{P_{(0, x)}} \big[ \psi (X_t) \big]. 
	\]
\end{theorem}

\begin{definition}[Uniform Strong Feller Family] \label{def: USFF}
A time inhomogeneous strong Markov family \(\{P_{ (s, x)} \colon (s, x) \in \bR_+ \times \bR^d\}\) is said to have the \emph{uniform strong Feller property} if, for every \(t > 0\) and every bounded Borel function \(\phi \colon \bR^d \to \bR\), the map \([0, t - h] \times \bR^d \ni (s, x) \mapsto E^{P_{(s, x)}}[ \phi (X_t) ]\) is uniformly continuous for every \(h \in (0, t)\). 
\end{definition}

The next result is the key observation for the proof of Theorem \ref{thm: main}. 

\begin{theorem} \label{thm: Markov = strong Feller}
Suppose that Condition \ref{SA: bounded and elliptic} holds.
Let \(\{P_{(s, x)} \colon (s, x) \in \bR_+ \times \mathbb{R}^d\}\) be a strong Markov family such that, for all \((s, x)\in \bR_+ \times \mathbb{R}^d\), \(P_{(s, x)} \in \cK (s, x)\). Then, it is also a uniform strong Feller family. 
\end{theorem}
\begin{proof}
We adapt the argument from \cite[Theorem 7.1.9]{SV}. 	Recall from \cite{SV} that a probability measure \(Q\) on \((\Omega, \cF)\) is said to be a \emph{solution to the martingale problem for \((0, a)\) starting from \((s, x) \in \bR_+ \times \bR^d\)} if \(Q( X_t = x \text{ for all } t \in [0, s]) = 1\) and the processes
\[
f (X_t) - \int_s^t  \tfrac{1}{2} \on{tr} \big[ a (X_r) \nabla^2 f (X_r) \big] dr \colon \ t \geq s, \ \ f \in C^\infty_c (\bR^d; \bR), 
\]
are \(Q\)-martingales. Thanks to \cite[Theorem 7.2.1]{SV}, for every \((s, x) \in \bR_+ \times \bR^d\), there exists a unique solution \(Q_{(s, x)}\) to the martingale problem for \((0, a)\) starting from \((s, x)\). 

By definition of the correspondence \(\cK\), we have \(P_{(s, x)} \in \fPas (s)\) and we denote the Lebesgue densities of the \(P_{(s, x)}\)-characteristics of the shifted coordinate process \(X_{\cdot + s}\) by \((b^{(s, x)}_{\cdot + s}, a^{(s, x)}_{\cdot + s})\). 
Notice that \((\llambda \otimes P_{(s, x)})\)-a.e. \(a^{(s, x)}_{\cdot + s} = a (X_{\cdot + s})\) by the definition of \(\Theta\) and \(\cK (s, x)\).
We define 
\[
Z^{(s, x)}_t := \exp \Big( - \int_s^{t \vee s} \langle a^{-1} (X_{r}) b^{(s, x)}_{r} , d \overline{X}^{(s, x)}_r \rangle - \tfrac{1}{2} \int_s^{t \vee s} \langle b^{(s, x)}_{r} , a^{-1} (X_{r}) b^{(s, x)}_{r} \rangle dr \Big), \quad t \in \bR_+,
\]
where
\[
\overline{X}^{(s, x)} := X - \int_s^{\cdot \vee s} b^{(s, x)}_{r} dr.
\]
Thanks to \cite[Lemma 6.4.1]{SV}, each \(Z^{(s, x)}\) is a \(P_{(s, x)}\)-martingale and 
\[
d Q_{(s, x)} = Z^{(s, x)}_T d P_{(s, x)} \text{ on } \cF_T \text{ for all } T \in \bR_+.
\]
Let \(\psi \colon \mathbb{R}^d \to \mathbb{R}\) be a bounded Borel function such that \(|\psi| \leq 1\), fix a finite time horizon \(T > 0\) and set
\[
\Psi (s,x) := E^{P_{(s, x)}} \big[ \psi (X_T) \big], \quad (s,x) \in \bR_+ \times \bR^d.
\]
Further, let \(h \in (0,T)\), pick \(0 \leq \alpha \leq h\) and let \((t,x), (s,y) \in [0,T-h] \times \bR^d\) such that \(t \leq s\).
Using the (strong) Markov property of \(\{P_{(s, x)} \colon (s, x) \in \bR_+ \times \mathbb{R}^d\}\), we obtain that
\begin{equation} \label{eq: 1st eps/3}
\begin{split}
    \big| E^{P_{(t, x)}} \big[ \psi (X_T) \big] &- E^{P_{(s, y)}} \big[ \psi (X_T) \big] \big| 
    \\&= \big| E^{P_{(t, x)}} \big[ \Psi (s+\alpha, X_{s+\alpha}) \big] - E^{P_{(s, y)}} \big[ \Psi (s + \alpha, X_{s+\alpha}) \big] \big|
    \\&\leq \big| E^{P_{(t, x)}} \big[ Z^{(t, x)}_{s+\alpha} \Psi (s+\alpha, X_{s+\alpha}) \big] - E^{P_{(s, y)}} \big[ Z^{(s, y)}_{s+\alpha} \Psi (s+\alpha, X_{s+\alpha}) \big] \big| 
    \\&\qquad + E^{P_{(t, x)}} \big[ |1 - Z^{(t, x)}_{s+\alpha}| \big] +  E^{P_{(s, y)}} \big[ |1 - Z^{(s, y)}_{s+\alpha}| \big]
    \\&= \big| E^{Q_{(t, x)}} \big[ \Psi (s+\alpha, X_{s+\alpha}) \big] - E^{Q_{(s, y)}} \big[ \Psi (s+\alpha, X_{s+\alpha}) \big] \big| 
    \\&\qquad + E^{P_{(t, x)}} \big[ |1 - Z^{(t, x)}_{s+\alpha}| \big] + E^{P_{(s, y)}} \big[ |1 - Z^{(s, y)}_{s+\alpha}| \big].
\end{split}
\end{equation}
Notice that
\begin{align*}
    \big( E^{P_{(t, x)}} \big[ |1 - Z^{(t, x)}_{s+\alpha}| \big] \big)^2 
    &\leq E^{P_{(t, x)}} \big[ \big|1 - Z^{(t, x)}_{s+\alpha}\big|^2 \big]
    \\&= E^{P_{(t, x)}} \big[ \big(Z^{(t, x)}_{s+\alpha} \big)^2 \big] - 1 
    \\&\leq e^{\C (s+\alpha - t)} - 1
    \\& = e^{\C\alpha} - 1 + e^{\C\alpha}\big(e^{\C (s - t)} - 1\big),
\end{align*}
where the constant \(\C > 0\) only depends on the constant from part (iii) of Condition~\ref{SA: bounded and elliptic}.
Similarly, we have
\[
\big( E^{P_{(s, y)}} \big[ |1 - Z^{(s, y)}_{s+\alpha}| \big] \big)^2 \leq e^{\C \alpha } - 1.
\]
Now, let \(\varepsilon > 0\) be arbitrary. Choose \(\alpha_\epsilon \in (0, h)\) with
\[
e^{\C \alpha_\epsilon } - 1 < \varepsilon^2.
\]
Next, choose \(\beta_\epsilon > 0\) with
\[
e^{\C \alpha_\epsilon} \big( e^{\C \beta_\epsilon}  - 1 \big) <  \varepsilon^2.
\]
This yields that
\begin{equation} \label{eq: part 1}
    E^{P_{(t, x)}} \big[ |1 - Z^{(t, x)}_{s+\alpha}| \big] + E^{P_{(s, y)}} \big[ |1 - Z^{(s, y)}_{s+\alpha}| \big] \leq 2 \epsilon + \epsilon = 3 \epsilon.
\end{equation}
Due to \cite[Theorem 7.2.4]{SV}, there exists an \(0 < \delta \leq \beta_\epsilon\) that does not depend on \((t,x), (s,y)\)
such that 
\begin{align} \label{eq: part 2}
\big| E^{Q_{(t, x)}} \big[ \Psi (s+\alpha_\epsilon, X_{s+\alpha_\epsilon}) \big] - E^{Q_{(s, y)}} \big[ \Psi (s+\alpha_\epsilon, X_{s+\alpha_\epsilon}) \big] \big| \leq \epsilon
\end{align}
whenever \(|t-s| + \| x - y \| \leq \delta\).
Here, we used that \( \alpha_\epsilon > 0\).
Thanks to \eqref{eq: part 1} and \eqref{eq: part 2},  we conclude that
\[
\big| E^{P_{(t, x)}} \big[ \psi (X_T) \big] - E^{P_{(s, y)}} \big[ \psi (X_T) \big] \big| \leq 4\varepsilon,
\]
in case \(|t-s| + \| x - y \| \leq \delta\).
This proves the uniform strong Feller property of the family \(\{P_{(s, x)} \colon (s, x) \in \bR_+ \times \mathbb{R}^d\}\) and therefore, the proof is complete.
\end{proof}

Finally, we are in the position to prove our main result, Theorem \ref{thm: main}.

\begin{proof} [Proof of Theorem \ref{thm: main}]
Let \(\psi \colon \mathbb{R}^d \to \mathbb{R}\) be bounded and upper semicontinuous and take \(t > 0\). By the Theorems \ref{theo: strong Markov selection} and \ref{thm: Markov = strong Feller}, there exists a uniform strong Feller family \(\{P_{(s, x)} \colon (s, x) \in \bR_+ \times \mathbb{R}^d\}\) such that 
\[
T_t (\psi) (x) = E^{P_{(0, x)}} \big[ \psi (X_t) \big], \quad x \in \bR^d.
\]
The uniform strong Feller property yields the uniform continuity of \(x \mapsto T_t (\psi)(x)\). This completes the proof.
\end{proof}


\end{document}